\documentclass[12pt]{amsart}
\usepackage{amssymb,amsfonts,amsmath,amsthm,cite,verbatim}
\usepackage{xcolor}
\usepackage{tikz}
\usepackage{graphicx}
\usepackage[left=2.5cm, right=2.5cm, top=3.5cm, bottom=3.5cm]{geometry}

\usepackage{hyperref}

\usepackage[normalem]{ulem}
\newtheorem{theorem}{Theorem}[section]

\newtheorem{prop}[theorem]{Proposition}
\newtheorem{lem}[theorem]{Lemma}

\newtheorem{cor}[theorem]{Corollary}

\makeatletter \@addtoreset{equation}{section}

\newcommand{\qbinom}[2]{\genfrac{[}{]}{0pt}{}{#1}{#2}}

\DeclareMathOperator*{\CT}{CT}

\begin{document}

\title[]{A recursion for a symmetric function generalization of the $q$-Dyson constant term identity}

\author{Yue Zhou}

\address{School of Mathematics and Statistics, Central South University,
Changsha 410075, P.R. China}

\email{zhouyue@csu.edu.cn}


\date{February 26, 2020}

\begin{abstract}
In 2000, Kadell gave an orthogonality conjecture for a symmetric function generalization of the $q$-Dyson constant term identity or the Zeilberger--Bressoud $q$-Dyson theorem.
The non-zero part of Kadell's orthogonality conjecture is a constant term identity indexed by a weak composition $v=(v_1,\dots,v_n)$ in the case when only one $v_i\neq 0$.
This conjecture was first proved by K\'{a}rolyi, Lascoux and Warnaar in 2015. They further formulated a closed-form expression for the above mentioned constant term in the case when all the parts of $v$ are distinct. Recently we obtain a recursion for this constant term provided that the largest part of $v$ occurs with multiplicity one in $v$. In this paper, we generalize our previous result to all compositions $v$.

\end{abstract}

\maketitle

\section{Introduction}\label{sec-intr}

In 1975, Andrews \cite{andrews1975} conjectured that for non-negative integers $a_1,a_2,\dots,a_n$,
\begin{equation}\label{q-Dyson}
\CT_x \prod_{1\leq i<j\leq n}
(x_i/x_j;q)_{a_i}(qx_j/x_i;q)_{a_j}=
\frac{(q;q)_{a_1+\cdots+a_n}}{(q;q)_{a_1}(q;q)_{a_2}\cdots(q;q)_{a_n}},
\end{equation}
where for $k$ a non-negative integer $(z;q)_k:=(1-z)(1-zq)\dots(1-zq^{k-1})$ is a
$q$-shifted factorial, and $\CT\limits_x$ denotes taking the constant term with respect to
$x:=(x_1,x_2,\dots,x_n)$. When $q\rightarrow 1^-$, the constant term identity~\eqref{q-Dyson} reduces to the Dyson constant term identity~\cite{dyson}.

Andrews' $q$-Dyson conjecture was first proved combinatorially in 1985 by Zeilberger and Bressoud \cite{zeil-bres1985}.
Twenty years later Gessel and Xin \cite{gess-xin2006} gave a second proof using formal Laurent series, and in 2014, K\'{a}rolyi and Nagy \cite{KN} discovered a very short and elegant proof using multivariable Lagrange interpolation.
Finally, Cai \cite{cai} found an inductive proof. These days, Andrews' ex-conjecture is usually referred as the Zeilberger--Bressoud $q$-Dyson theorem or the $q$-Dyson constant term identity.

There are many generalizations of the equal parameter case of the $q$-Dyson constant term identity, i.e., $a_1=a_2=\cdots=a_n$. For example, the $q$-Morris identity \cite{Morris1982} or the $q$-Selberg integral \cite{Askey,FW,Selberg}, and the Macdonald constant term ex-conjecture \cite{cherednik,MacSMC,Mac95,opdam}.
In the theory of Macdonald polynomials \cite{Mac95}, the equal parameter case of the $q$-Dyson identity is equivalent to a scalar product identity. This provides a satisfactory explanation for the equal parameter case of the $q$-Dyson identity in terms of orthogonal polynomials. Giving a similar such explanation for the full $q$-Dyson identity is an important open problem.

The first step towards a resolution of this problem was made by Kadell \cite{kadell}.
He formulated a symmetric function generalization of the $q$-Dyson identity and gave an orthogonal conjecture we will describe next.
Let $X=(x_1,x_2,\dots)$ be an alphabet of countably many variables.
Then the $r$th complete symmetric function $h_r(X)$
may be defined in terms of its generating function as
\begin{equation}\label{e-gfcomplete}
\sum_{r\geq 0} z^r h_r(X)=\prod_{i\geq 1}
\frac{1}{1-zx_i}.
\end{equation}
More generally, for the complete symmetric function indexed by a composition (or partition)
$v=(v_1,v_2,\dots,v_k)$
\[
h_v:=h_{v_1}\cdots h_{v_k}.
\]
For $a:=(a_1,a_2,\dots,a_n)$ a sequence of non-negative integers, let $x^{(a)}$ denote the alphabet
\[
x^{(a)}:=(x_1,x_1q,\dots,x_1q^{a_1-1},\dots,
x_n,x_nq,\dots,x_nq^{a_n-1})
\]
of cardinality $|a|:=a_1+\cdots+a_n$.
Define a generalized $q$-Dyson constant term
\begin{equation}\label{GDyson1}
D_{v,\lambda}(a):=\CT_x
x^{-v}h_{\lambda}\big(x^{(a)}\big)
\prod_{1\leq i<j\leq n}
(x_i/x_j;q)_{a_i}(qx_j/x_i;q)_{a_j}.
\end{equation}
Here $v=(v_1,v_2,\dots,v_n)\in\mathbb{Z}^{n}$,
$x^v$ denotes the monomial $x_1^{v_1}\cdots x_n^{v_n}$,
and $\lambda$ is a partition such that
$|v|=|\lambda|$. (Note that if $|v|\neq|\lambda|$ then
$D_{v,\lambda}(a)=0$.)
For the constant term \eqref{GDyson1}, Kadell \cite[Conjecture 4]{kadell}
conjectured that for $r$ a positive integer and $v$ a composition such that $|v|=r$,
\begin{equation}\label{kadellconj}
D_{v,(r)}(a)=
\begin{cases}
\displaystyle
q^{\sum_{i=k+1}^n a_i}\qbinom{|a|+r-1}{a_k-1}
\prod_{\substack{1\leq i\leq n\\i\neq k}}\qbinom{a_i+\cdots+a_n}{a_i}
& \text{if $v=(0^{k-1},r,0^{n-k})$}, \\[6mm]
0 & \text{otherwise},
\end{cases}
\end{equation}
where $\qbinom{n}{k}$ is a $q$-binomial coefficient.
In fact Kadell only considered $v=(r,0^{n-1})$ in his conjecture, but the
more general statement given above is what was proved by K\'{a}rolyi,
Lascoux and Warnaar in \cite[Theorem 1.3]{KLW} using multivariable Lagrange
interpolation and key polynomials.
If for a sequence $u=(u_1,\dots,u_n)$ of integers we denote
by $u^{+}$ the sequence obtained from $u$ by ordering the $u_i$
in weakly decreasing order (so that $u^{+}$ is a partition
if $u$ is a composition), then K\'{a}rolyi et al.\ also
proved a closed-form expression for
$D_{v,v^{+}}(a)$ in the case when $v$ is a composition all of whose parts
are distinct, i.e., $v_i\neq v_j$ for all $1\leq i<j\leq n$.
Subsequently, Cai~\cite{cai} gave an inductive proof of Kadell's conjecture.
Recently, we~\cite{Zhou} obtained a recursion for $D_{v,v^+}(a)$ if
$v$ is a composition such that its largest part has multiplicity one,
see Corollary~\ref{cor-1} below.

In this paper, we obtain a recursion for $D_{v,v^+}(a)$ for $v$ an arbitrary non-zero composition.
Given a sequence $s=(s_1,\dots,s_n)$ and an integer $k\in \{1,2,\dots,n\}$,
define
$s^{(k)}:=(s_1,\dots,s_{k-1},s_{k+1},\dots,s_n)$.
Furthermore, for a subset $I=\{k_1,k_2,\dots,k_i\}$ of $\{1,2,\dots,n\}$, define
$s^{(I)}:=(s_1,s_2,\dots,\widehat{s}_{k_1},\dots,\widehat{s}_{k_2},\dots,\widehat{s}_{k_i},\dots,s_n)$,
where $\widehat{t}$ denotes the omission of $t$.
\begin{theorem}\label{Thm-rec}
Let $D_{0,0}(0):=1$. Then for $v=(v_1,\dots,v_n)$ a non-zero composition such that its largest part equals $r$ and $I:=\big\{i\in \{1,2,\dots,n\}: v_i=r\big\}$,
the following recursion holds:
\begin{equation}\label{formuh1}
D_{v,v^+}(a)=D_{v^{(I)},(v^{(I)})^+}\big(a^{(I)}\big)\frac{\qbinom{|a|+r-1}{a,r-1}}
{\qbinom{|a|-a_I+r-1}{a^{(I)},r-1}}
\sum_{\emptyset \neq J\subseteq I}(-1)^{|I\setminus J|}q^{L_{I,J}(a)}
\frac{1-q^{a_{J}}}{1-q^{|a|-a_{J}+r}},
\end{equation}
where $|S|$ denotes the number of elements in the set $S$,
$a_S:=\sum_{j\in S}a_j$, $\qbinom{n}{c}$ is a $q$-multinomial coefficient for the composition $c$, and
\begin{equation}\label{def-L}
L_{I,J}(a)=\sum_{\substack{1\leq i\leq j\leq n\\ i\in I, j\notin J}}a_j.
\end{equation}
\end{theorem}
For example, if $v=(0,2,3,2,3,1)$, then $v^+=(3,3,2,2,1,0)$, $I=\{3,5\}$ and
$v^{(I)}=(0,2,2,1)$.
For $v$ a composition the constant term $D_{v,\lambda}(a)=0$ if $a_i=0$ and $v_i\neq 0$ for some $i$. On the other hand, if $a_i=v_i=0$ for some $i$,
then the constant term $D_{v,\lambda}(a)$ reduces to the $n-1$ variable case $D_{v^{(i)},\lambda}(a^{(i)})$. Hence, since we only concern the constant term $D_{v, v^+}(a)$ for $v$ a composition in this paper, we can assume that all the $a_i$ are positive integers from now on.
We note that for $v=(0,\dots,0)$
the constant term $D_{v,v^+}(a)$ corresponds to the $q$-Dyson constant term (the left-hand side of \eqref{q-Dyson}).
Using the recursion~\eqref{formuh1} and the $q$-Dyson identity~\eqref{q-Dyson},
we can obtain a closed-form formula for $D_{v,v^+}(a)$ for arbitrary compositions $v$.
If the largest part of $v$ has multiplicity one in $v$, then
Theorem~\ref{Thm-rec} reduces to \cite[Theorem~1.3]{Zhou}.
\begin{cor}\label{cor-1}
Let $v=(v_1,\dots,v_n)$ be a composition
such that its largest part has multiplicity one in $v$.
Fix a positive integer $k$ by $v_k=\max\{v\}$.
Then
\begin{equation}\label{e-Dyson}
D_{v,v^{+}}(a)=q^{\sum_{i=k+1}^n a_i}
\qbinom{v_k+|a|-1}{a_k-1}
D_{v^{(k)},(v^{(k)})^{+}}\big(a^{(k)}\big).
\end{equation}
\end{cor}

The method employed to prove Theorem~\ref{Thm-rec}
is based on Cai's splitting formula (see \eqref{e-Fsplit} below) for the following rational function.
\begin{equation}\label{d-F}
F(a,w):=\prod_{1\leq i<j\leq n}
(x_i/x_j;q)_{a_i}(qx_j/x_i;q)_{a_j}
\prod_{i=1}^n\prod_{j=1}^s(x_i/w_j;q)_{a_i}^{-1},
\end{equation}
where $w:=(w_1,\dots,w_s)$ is a sequence of parameters.
Throughout this paper, we assume that all terms of the form $cx_i/w_j$ in \eqref{d-F}
satisfy $|cx_i/w_j|<1$, where $c\in \mathbb{C}(q)\setminus \{0\}$. Hence,
\[
\frac{1}{1-cx_i/w_j}=\sum_{k\geq 0} (cx_i/w_j)^k.
\]
By the generating function of complete symmetric functions \eqref{e-gfcomplete},
the constant term $D_{v,\lambda}(a)$ equals a certain coefficient of $F(a,w)$. That is
\begin{equation}\label{e-relation}
D_{v,\lambda}(a)=\CT_{x,w}x^{-v}w^{\lambda}F(a,w),
\end{equation}
where we assume $s$ in \eqref{d-F} is fixed by $\ell(\lambda)$ the length of $\lambda$.
In \cite{cai}, Cai gave a splitting formula for $F(a,w)$. Using his formula, we obtain an inductive formula for $D_{v,v^+}(a)$ (see Lemma~\ref{lem-ind}).
Then, by the inductive formula we prove the recursion \eqref{formuh1} for $D_{v,v^+}(a)$.

The remainder of this paper is organised as follows.
In the next section we introduce some basic
notation used throughout this paper.
In Section~\ref{sec-splitting}, we present Cai's splitting formula for $F(a,w)$.
In Section~\ref{sec-inductive}, we obtain an inductive formula for $D_{v,v^+}(a)$ using Cai's splitting formula.
In Section~\ref{sec-preliminaries}, we prepare some results used in the proof of Theorem~\ref{Thm-rec}.
In Section~\ref{sec-proof}, we complete the proof of the recursion for $D_{v,v^+}(a)$.

\section{Basic notation}\label{sec-notation}

In this section we introduce some basic notation
used throughout this paper.

For $v=(v_1,\dots,v_n)$ a sequence, we write
$|v|$ for the sum of its entries, i.e.,
$|v|=v_1+\cdots+v_n$.
Moreover, if $v\in\mathbb{R}^{n}$ then we write
$v^{+}$ for the sequence obtained from $v$
by ordering its elements in weakly decreasing order.
If all the entries of $v$ are non-negative integers, we refer to $v$ as a (weak) composition.
A partition is a sequence $\lambda={(\lambda_1,\lambda_2,\dots)}$ of non-negative integers such that
${\lambda_1\geq \lambda_2\geq \cdots}$ and
only finitely-many $\lambda_i$ are positive.
The length of a partition $\lambda$, denoted
$\ell(\lambda)$ is defined to be the number of non-zero $\lambda_i$ (such $\lambda_i$ are known as the parts of $\lambda$).
We adopt the convention of not displaying the
tails of zeros of a partition.

For $k$ a non-negative integer, the $q$-shifted factorial is defined as
\[
(z)_k=(z;q)_k:=\prod_{i=0}^{k-1}(1-zq^i),
\]
where, typically, we suppress the base $q$.
Using the above we can define the
$q$-binomial coefficient as
\[
\qbinom{n}{k}=\frac{(q^{n-k+1})_k}{(q)_k}
\]
for $n$ and $k$ non-negative integers.
Furthermore, for $n$ a non-negative integer and $s=(s_1,s_2,\dots,s_k)$ a composition such that $|s|=n$,
we define the $q$-multinomial coefficient as
\[
\qbinom{n}{s}=\frac{(q)_n}{(q)_{s_1}(q)_{s_2}\cdots(q)_{s_k}}.
\]
In particular,
\[
\qbinom{n}{(k,n-k)}=\qbinom{n}{k}.
\]

\section{A splitting formula for $F(a,w)$}\label{sec-splitting}

In this section, we present Cai's splitting formula \cite[Proposition 4.2]{cai} for $F(a,w)$. To make the paper self-contained, we include a proof of this splitting formula using partial fraction decomposition.

We need the following simple result in the proof of the splitting formula mentioned above.
\begin{lem}\label{prop-c}
Let $i$ and $j$ be positive integers. Then, for $k$ an integer such that $0\leq k\leq j-1$,
\begin{subequations}\label{e-ab}
\begin{equation}\label{prop-b2}
\frac{(1/z)_{i}(qz)_j}{(q^{-k}/z)_i}=q^{ik}(q^{1-i}z)_k(q^{k+1}z)_{j-k},
\end{equation}
and
\begin{equation}\label{prop-a}
\frac{(z)_j(q/z)_i}{(q^{-k}/z)_i}=q^{(k+1)i}(q^{-i}z)_{k+1}(q^{k+1}z)_{j-k-1}.
\end{equation}
\end{subequations}
\end{lem}
Note that the $j=k$ case of \eqref{prop-b2} (taking $z\mapsto z/q$) is the standard fact in \cite[Equation~(I.13)]{GR}.
\begin{proof}
For $0\leq k\leq j$,
\[
\frac{(1/z)_{i}(qz)_j}{(q^{-k}/z)_i}=\frac{(q^{i-k}/z)_k(qz)_{j}}{(q^{-k}/z)_{k}}
=\frac{(-1/z)^kq^{ik-\binom{k+1}{2}}(q^{1-i}z)_k(qz)_j}{(-1/z)^{k}q^{-\binom{k+1}{2}}(qz)_k}
=q^{ik}(q^{1-i}z)_k(q^{k+1}z)_{j-k}.
\]
Taking $z\mapsto z/q$ and $k\mapsto k+1$ in \eqref{prop-b2} yields
\eqref{prop-a} for $-1\leq k\leq j-1$.
Hence, both \eqref{prop-b2} and \eqref{prop-a} hold simultaneously for $0\leq k\leq j-1$.
\end{proof}

In \cite{cai}, Cai showed that $F(a,w)$ admits the following partial fraction expansion.
\begin{prop}
Let $F(a,w)$ be defined as in \eqref{d-F}. Then
\begin{equation}\label{e-Fsplit}
F(a,w)=\sum_{i=1}^{n}\sum_{j=0}^{a_i-1}\frac{A_{ij}}{1-q^jx_i/w_1},
\end{equation}
where
\begin{multline}\label{A}
A_{ij}=\frac{1}
{(q^{-j})_j(q)_{a_i-j-1}}\prod_{\substack{1\leq v<u\leq n\\v,u\neq i}}
(x_v/x_u)_{a_v}(qx_u/x_v)_{a_u}
\prod_{u=1}^n\prod_{v=2}^s(x_u/w_v)_{a_u}^{-1}
\\
\times
\prod_{l=1}^{i-1}q^{ja_l}\big(q^{1-a_l}x_i/x_l\big)_j\big(q^{j+1}x_i/x_l\big)_{a_i-j}
\prod_{l=i+1}^{n}q^{(j+1)a_l}\big(q^{-a_l}x_i/x_l\big)_{j+1}\big(q^{j+1}x_i/x_l\big)_{a_i-j-1}.
\end{multline}
\end{prop}
Note that $A_{ij}$ is a power series in $x_i$.
\begin{proof}
By partial fraction decomposition of $F(a,w)$ with respect to $w_1$,
we can rewrite $F(a,w)$ as \eqref{e-Fsplit} and
\begin{equation}\label{subs-b}
A_{ij}=F(a,w)(1-q^jx_i/w_1)|_{w_1=q^jx_i}.
\end{equation}
Carrying out the substitution $w_1=q^jx_i$ in $F(a,w)(1-q^jx_i/w_1)$ yields
\begin{multline}\label{e-Bij}
A_{ij}=
\frac{1}{(q^{-j})_j(q)_{a_i-j-1}}
\prod_{l=1}^{i-1}\frac{(x_l/x_i)_{a_l}(qx_i/x_l)_{a_i}}{(q^{-j}x_l/x_i)_{a_l}}
\prod_{l=i+1}^n\frac{(x_i/x_l)_{a_i}(qx_l/x_i)_{a_l}}{(q^{-j}x_l/x_i)_{a_l}}
\\
\times
\prod_{\substack{1\leq v<u\leq n\\v,u\neq i}}
(x_v/x_u)_{a_v}(qx_u/x_v)_{a_u}
\prod_{u=1}^n\prod_{v=2}^s(x_u/w_v)_{a_u}^{-1}.
\end{multline}

Using \eqref{e-ab} with $(i,j,k,z)\mapsto (a_l,a_i,j,x_i/x_l)$,
we have
\begin{subequations}\label{e-B}
\begin{equation}
\frac{(x_l/x_i)_{a_l}(qx_i/x_l)_{a_i}}{(q^{-j}x_l/x_i)_{a_l}}
=q^{ja_l}\big(q^{1-a_l}x_i/x_l\big)_j\big(q^{j+1}x_i/x_l\big)_{a_i-j},
\end{equation}
and
\begin{equation}
\frac{(x_i/x_l)_{a_i}(qx_l/x_i)_{a_l}}{(q^{-j}x_l/x_i)_{a_l}}
=q^{(j+1)a_l}\big(q^{-a_l}x_i/x_l\big)_{j+1}\big(q^{j+1}x_i/x_l\big)_{a_i-j-1}
\end{equation}
respectively.
\end{subequations}
Substituting \eqref{e-B} into \eqref{e-Bij} we obtain \eqref{A}.
\end{proof}

\section{An inductive formula for $D_{v,v^+}(a)$}\label{sec-inductive}

In this section, we utilize Cai's splitting formula \eqref{e-Fsplit} for $F(a,w)$
to obtain an inductive formula for $D_{v,v^+}(a)$.

We begin by giving an easy result deduced from the $q$-binomial theorem.
\begin{prop}\label{prop-sum}
Let $n$ and $t$ be non-negative integers. Then
\begin{equation}\label{e-sum}
\sum_{k=0}^{t}
\frac{q^{k(n-t)}}
{(q^{-k})_{k}(q)_{t-k}}
=\qbinom{n}{t}.
\end{equation}
\end{prop}
\begin{proof}
We can rewrite the left-hand side of \eqref{e-sum} as
\begin{equation}\label{e-sum2}
\sum_{k=0}^{t}
\frac{(-1)^kq^{k(n-t)+\binom{k+1}{2}}}{(q)_{k}(q)_{t-k}}
=\frac{1}{(q)_t} \sum_{k=0}^{t}q^{\binom{k}{2}}\qbinom{t}{k}(-q^{n-t+1})^k.
\end{equation}
Using the well-known $q$-binomial theorem \cite[Theorem 3.3]{andrew-qbinomial}
\[
(z)_t=\sum_{k=0}^tq^{\binom{k}{2}}\qbinom{t}{k}(-z)^k
\]
with $z\mapsto q^{n-t+1}$, we have
\[
\frac{1}{(q)_t} \sum_{k=0}^{t}q^{\binom{k}{2}}\qbinom{t}{k}(-q^{n-t+1})^k
=(q^{n-t+1})_t/(q)_t=\qbinom{n}{t}. \qedhere
\]
\end{proof}

By the splitting formula for $F(a,w)$, we obtain the following inductive formula for
$D_{v,v^+}(a)$.
\begin{lem}\label{lem-ind}
Let $v=(v_1,\dots,v_n)$ be a non-zero composition such that its largest part equals $r$, and set $I:=\big\{i\in \{1,2,\dots,n\}: v_i=r\big\}$.
Then
\begin{equation}\label{e-ind}
D_{v,v^+}(a)=\sum_{i\in I} q^{\sum_{j=i+1}^na_j}\qbinom{|a|+r-1}{a_{i}-1}D_{v^{(i)},(v^{(i)})^+}\big(a^{(i)}\big).
\end{equation}
\end{lem}
\begin{proof}
Substituting the splitting formula \eqref{e-Fsplit} for $F(a,w)$ into
\begin{equation}\label{e-trans}
D_{v,v^+}(a)=\CT_{x,w} \frac{w^{v^+}}{x^v}F(a,w),
\end{equation}
we have
\[
D_{v,v^+}(a)=\sum_{i=1}^{n}\sum_{j=0}^{a_i-1}\CT_{x,w}\frac{w^{v^+}}{x^v}
\frac{A_{ij}}{1-q^jx_i/w_1},
\]
where we take $w=(w_1,\dots,w_n)$.
Taking the constant term with respect to $w_1$, we obtain
\begin{equation}\label{e-ind1}
D_{v,v^+}(a)=\sum_{i=1}^{n}\sum_{j=0}^{a_i-1}\CT_{x,w^{(1)}}\frac{q^{j\mu_1}x_i^{\mu_1}w_2^{\mu_2}
\cdots w_n^{\mu_n}}{x^v}A_{ij}
=\sum_{i=1}^{n}\sum_{j=0}^{a_i-1}\CT_{x,w^{(1)}}\frac{q^{jr}x_i^{r}w_2^{\mu_2}
\cdots w_n^{\mu_n}}{x^v}A_{ij},
\end{equation}
where $\mu:=v^+=(\mu_1,\dots,\mu_n)$ and $\mu_1=\max\{v\}=r$.
Since $A_{ij}$ is a power series in $x_i$ for all $i$,
\[
\CT_{x,w^{(1)}}\frac{q^{jr}x_i^rw_2^{\mu_2}\cdots w_n^{\mu_n}}
{x^v}A_{ij}=0 \quad \text{for $i\notin I$,}
\]
and
\[
\CT_{x,w^{(1)}}\frac{q^{jr}x_i^rw_2^{\mu_2}\cdots w_n^{\mu_n}}
{x^v}A_{ij}=\CT_{x^{(i)},w^{(1)}}\frac{q^{jr}w_2^{\mu_2}\cdots w_n^{\mu_n}}
{x_1^{v_1}\cdots x_{i-1}^{v_{i-1}}x_{i+1}^{v_{i+1}}\cdots x_n^{v_n}}A_{ij}|_{x_i=0}
\quad \text{for $i\in I$.}
\]
Hence, \eqref{e-ind1} reduces to
\begin{equation*}
D_{v,v^+}(a)=\sum_{i\in I}\sum_{j=0}^{a_i-1}\CT_{x^{(i)},w^{(1)}}
\frac{q^{jr}w_2^{\mu_2}\cdots w_n^{\mu_n}}
{x_1^{v_1}\cdots x_{i-1}^{v_{i-1}}x_{i+1}^{v_{i+1}}\cdots x_n^{v_n}}A_{ij}|_{x_i=0}.
\end{equation*}
By the expression \eqref{A} for $A_{ij}$ and by carrying out the substitution $x_i=0$ in $A_{ij}$, we obtain
\begin{multline*}
D_{v,v^+}(a)=\sum_{i\in I}\sum_{j=0}^{a_i-1}\frac{q^{j(a_1+\cdots+a_{i-1}+r)+(j+1)(a_{i+1}+\cdots+a_n)}}
{(q^{-j})_j(q)_{a_i-j-1}}\\
\times\CT_{x^{(i)},w^{(1)}}\frac{w_2^{\mu_2}\cdots w_n^{\mu_n}}{x_1^{v_1}\cdots x_{i-1}^{v_{i-1}}x_{i+1}^{v_{i+1}}\cdots x_n^{v_n}}
\prod_{\substack{1\leq v<u\leq n\\v,u\neq i}}
(x_v/x_u)_{a_v}(qx_u/x_v)_{a_u}
\prod_{\substack{u=1\\u\neq i}}^n\prod_{v=2}^n(x_u/w_v)_{a_u}^{-1}.\\
\end{multline*}
Using \eqref{e-sum} with $k\mapsto j, t\mapsto a_i-1$ and $n\mapsto |a|+r-1$,
we find
\begin{multline*}
D_{v,v^+}(a)=
\sum_{i\in I}q^{a_{i+1}+\cdots+a_n}\qbinom{|a|+r-1}{a_i-1}
\CT_{x^{(i)},w^{(1)}}
\frac{w_2^{\mu_2}\cdots w_n^{\mu_n}}
{x_1^{v_1}\cdots x_{i-1}^{v_{i-1}}x_{i+1}^{v_{i+1}}\cdots x_n^{v_n}}\\
\times\prod_{\substack{1\leq v<u\leq n\\v,u\neq i}}
(x_v/x_u)_{a_v}(qx_u/x_v)_{a_u}
\prod_{\substack{u=1\\u\neq i}}^n\prod_{v=2}^n(x_u/w_v)_{a_u}^{-1}.
\end{multline*}
By \eqref{e-trans} again with $a\mapsto a^{(i)}, v\mapsto v^{(i)}, x\mapsto x^{(i)}$ and
$w\mapsto w^{(1)}$, the constant term in the above sum equals $D_{v^{(i)},(v^{(i)})^+}\big(a^{(i)}\big)$,
completing the proof.
\end{proof}

\section{Preliminaries for the proof of Theorem~\ref{Thm-rec}}\label{sec-preliminaries}

In this section, we prepare some results used in the proof of Theorem~\ref{Thm-rec}.

\begin{prop}
Let $I$ be a non-empty subset of $\{1,2,\dots,n\}$ and $J$ be a non-empty subset of $I$.
Let $L_{I, J}(a)$ be defined as in \eqref{def-L}. Then, for an element $i\in J$,
\begin{equation}\label{e-keyeq}
L_{I\setminus\{i\}, J\setminus\{i\}}\big(a^{(i)}\big)=L_{I, J}(a)-\sum_{j=i}^na_j
+\sum_{\substack{j\in J\\j\geq i}}a_j.
\end{equation}
\end{prop}
Note that if $I$ is a one-element subset of $\{1,2,\dots,n\}$, then both sides of \eqref{e-keyeq} vanish.
\begin{proof}
By the definition of $L_{I,J}(a)$ in \eqref{def-L} with $I\mapsto I\setminus\{i\}$,
$J\mapsto J\setminus\{i\}$ and $a\mapsto a^{(i)}$, we have
\[
L_{I\setminus\{i\},J\setminus\{i\}}\big(a^{(i)}\big)
=\sum_{\substack{1\leq u\leq v\leq n\\[1pt] u\in I\setminus\{i\}, v\notin J\setminus\{i\},v\neq i}}a_v.
\]
Hence, for $i\in J$
\[
L_{I\setminus\{i\},J\setminus\{i\}}\big(a^{(i)}\big)=
\sum_{\substack{1\leq u\leq v\leq n\\[1pt] u\in I\setminus\{i\}, v\notin J}}a_v
=L_{I\setminus\{i\},J}(a).
\]
We can rewrite $L_{I\setminus\{i\},J}(a)$ as
\[
L_{I,J}(a)-\sum_{\substack{j=i\\j\notin J}}^na_j.
\]
It is not hard to see that this equals the right-hand side of \eqref{e-keyeq}.
\end{proof}

The next two lemmas concern sums related to $L_{I,J}(a)$.
In these sums, large cancellations occur and the sums reduce to one or two terms.
\begin{lem}\label{lem-a}
Let $I$ be a subset of $\{1,2,\dots,n\}$ of cardinality at least two.
Then, for an element $i\in I$,
\begin{equation}\label{e-1}
\sum_{\emptyset \neq J\subseteq I\setminus \{i\}}
(-1)^{|J|}q^{L_{I\setminus\{i\},J}(a^{(i)})+a_{J}}
=-q^{L_{I,\{i\}}(a)-\sum_{j=i+1}^na_j},
\end{equation}
where $a_S:=\sum_{j\in S}a_j$ and $L_{I,J}(a)$ is defined as in \eqref{def-L}.
\end{lem}
\begin{proof}
Denote by $k$ the least element in $I\setminus\{i\}$.
We can then rewrite the left-hand side of \eqref{e-1} as
\begin{equation}\label{e-a}
\sum_{\substack{J\subseteq I\setminus \{i\}
\\J \neq \emptyset \ \text{and}\ k\notin J}}
\Big((-1)^{|J|}q^{L_{I\setminus\{i\},J}(a^{(i)})+a_{J}}
+(-1)^{|J\cup \{k\}|}q^{L_{I\setminus\{i\},J\cup \{k\}}(a^{(i)})+a_{J\cup \{k\}}}
\Big)
-q^{L_{I\setminus\{i\},\{k\}}(a^{(i)})+a_k}.
\end{equation}
Note that if there is no $J\subseteq I\setminus \{i\}$ such that $J\neq \emptyset$
and $k\notin J$, then we set the above sum to be zero.
We complete the proof by showing that the sum in \eqref{e-a} equals zero and
\begin{equation}\label{e-trans2}
L_{I\setminus\{i\},\{k\}}\big(a^{(i)}\big)+a_k=L_{I,\{i\}}(a)-\sum_{j=i+1}^na_j.
\end{equation}

To show that the sum in \eqref{e-a} equals zero, it suffices to show that
\begin{equation}\label{e-cupk}
L_{I\setminus\{i\},J}\big(a^{(i)}\big)+a_{J}=L_{I\setminus\{i\},J\cup \{k\}}\big(a^{(i)}\big)+a_{J\cup \{k\}}
\end{equation}
for $J\subseteq I\setminus \{i\}$, $J\neq \emptyset$ and $k\notin J$.
Since $J\subseteq I\setminus \{i\}$, $k\notin J$ and $k$ is the least element in $I\setminus\{i\}$, the integer
$k$ is smaller than any element of $J$. Together with the definition of $L_{I,J}(a)$ in \eqref{def-L}, we find
\[
L_{I\setminus\{i\},J\cup \{k\}}\big(a^{(i)}\big)=L_{I\setminus\{i\},J}\big(a^{(i)}\big)-a_k.
\]
Substituting this into the right-hand side of \eqref{e-cupk} and using
$a_{J\cup \{k\}}=a_{J}+a_k$, we obtain the left-hand side of \eqref{e-cupk}.

Since $k$ is the least element in $I\setminus\{i\}$, by the definition of $L_{I,J}(a)$
\[
L_{I\setminus\{i\},\{k\}}\big(a^{(i)}\big)+a_k
=L_{I\setminus\{i\},\emptyset}\big(a^{(i)}\big).
\]
It is easy to see that $L_{I\setminus\{i\},\emptyset}\big(a^{(i)}\big)$
can be written as $L_{I,\{i\}}(a)-\sum_{j=i+1}^na_j$.
Hence, $\eqref{e-trans2}$ holds.
\end{proof}

\begin{lem}\label{lem-b}
Let $I$ be a non-empty subset of $\{1,2,\dots,n\}$ and $J$ be a non-empty subset of $I$.
Then
\begin{equation}\label{e-2}
\sum_{i\in J} (-1)^{|J\setminus \{i\}|+1}q^{\sum_{j=i+1}^na_j+L_{I\setminus\{i\},J\setminus \{i\}}(a^{(i)})}
(1-q^{a_i})
=(-1)^{|J|}q^{L_{I,J}(a)}(1-q^{a_J}),
\end{equation}
where $a_S:=\sum_{j\in S}a_j$ and $L_{I,J}(a)$ is defined as in \eqref{def-L}.
\end{lem}
\begin{proof}
We prove \eqref{e-2} by telescoping.
Let $J=\{k_1,k_2,\dots,k_s\}$ such that $1\leq k_1<k_2<\cdots<k_s\leq n$.
Write the left-hand side of \eqref{e-2} as
\begin{equation}\label{e-s}
\sum_{i=1}^st_{k_i}(1-q^{a_{k_i}}),
\end{equation}
where
\begin{equation}\label{e-defi-t}
t_{k_i}:=(-1)^{|J\setminus \{k_i\}|+1}q^{\sum_{j=k_i+1}^na_j+L_{I\setminus\{k_i\},J\setminus \{k_i\}}(a^{(k_i)})}.
\end{equation}
We will show that
\begin{equation}\label{e-telescoping}
t_{k_{i-1}}=q^{a_{k_i}}t_{k_i} \quad \text{for $i=2,\dots,s$,}
\end{equation}
and the sum \eqref{e-s} reduces to
\[
t_{k_s}-q^{a_{k_1}}t_{k_1}.
\]
Then we will show that this equals the right-hand side of \eqref{e-2}.

Since $|J\setminus \{k_{i-1}\}|=|J\setminus \{k_i\}|=|J|-1$, to show $t_{k_{i-1}}=q^{a_{k_i}}t_{k_i}$ for $i=2,\dots,s$, it suffices to show that
for a fixed integer $i$ such that $2\leq i\leq s$,
\begin{equation}\label{e-3}
\sum_{j=k_{i-1}+1}^na_j+L_{I\setminus\{k_{i-1}\},J\setminus \{k_{i-1}\}}\big(a^{(k_{i-1})}\big)
=\sum_{j=k_{i}+1}^na_j+L_{I\setminus\{k_i\},J\setminus \{k_i\}}\big(a^{(k_i)}\big)+a_{k_i}.
\end{equation}
By \eqref{e-keyeq}
\begin{align*}
\sum_{j=k_{i-1}+1}^na_j+L_{I\setminus\{k_{i-1}\},J\setminus \{k_{i-1}\}}
\big(a^{(k_{i-1})}\big)
&=\sum_{j=k_{i-1}+1}^na_j+L_{I,J}(a)-\sum_{j=k_{i-1}}^na_j
+\sum_{\substack{j\in J\\j\geq k_{i-1}}}a_j\\
&=L_{I,J}(a)-a_{k_{i-1}}+\sum_{\substack{j\in J\\j\geq k_{i-1}}}a_j
=L_{I,J}(a)+\sum_{\substack{j\in J\\j\geq k_i}}a_j.
\end{align*}
Using \eqref{e-keyeq} again, we find
\begin{align*}
L_{I,J}(a)+\sum_{\substack{j\in J\\j\geq k_i}}a_j
&=L_{I\setminus\{k_i\},J\setminus \{k_i\}}\big(a^{(k_i)}\big)
+\sum_{j=k_i}^na_j-\sum_{\substack{j\in J\\j\geq k_i}}a_j
+\sum_{\substack{j\in J\\j\geq k_i}}a_j\\
&=L_{I\setminus\{k_i\},J\setminus \{k_i\}}\big(a^{(k_i)}\big)
+\sum_{j=k_i}^na_j
=L_{I\setminus\{k_i\},J\setminus \{k_i\}}\big(a^{(k_i)}\big)
+\sum_{j=k_i+1}^na_j+a_{k_i}.
\end{align*}
Hence, we obtain \eqref{e-3} and \eqref{e-telescoping} follows.

As mentioned above, by \eqref{e-telescoping}
\[
\sum_{i=1}^st_{k_i}(1-q^{a_{k_i}})
=t_{k_s}-q^{a_{k_1}}t_{k_1}.
\]
We complete the proof by showing that
\begin{equation*}
t_{k_s}-q^{a_{k_1}}t_{k_1}=(-1)^{|J|}q^{L_{I,J}(a)}(1-q^{a_J}).
\end{equation*}
By the expression \eqref{e-defi-t} for $t_{k_i}$, we have
\begin{multline*}
t_{k_s}-q^{a_{k_1}}t_{k_1}
=(-1)^{|J\setminus \{k_s\}|+1}q^{\sum_{j=k_s+1}^na_j+L_{I\setminus\{k_s\},J\setminus \{k_s\}}(a^{(k_s)})}\\
-(-1)^{|J\setminus \{k_1\}|+1}q^{\sum_{j=k_1+1}^na_j+L_{I\setminus\{k_1\},J\setminus \{k_1\}}(a^{(k_1)})+a_{k_1}}.
\end{multline*}
By \eqref{e-keyeq}
\begin{multline}\label{e-5}
t_{k_s}-q^{a_{k_1}}t_{k_1}
=(-1)^{|J\setminus \{k_s\}|+1}q^{\sum_{j=k_s+1}^na_j+L_{I,J}(a)
-\sum_{j=k_s}^na_j+\sum_{\substack{j\in J\\j\geq k_s}}a_j}\\
-(-1)^{|J\setminus \{k_1\}|+1}q^{\sum_{j=k_1+1}^na_j+L_{I,J}(a)
-\sum_{j=k_1}^na_j+\sum_{\substack{j\in J\\j\geq k_1}}a_j+a_{k_1}}.
\end{multline}
Using
\[
|J\setminus \{k_s\}|+1=|J\setminus \{k_1\}|+1=|J|, \quad
\sum_{\substack{j\in J\\j\geq k_s}}a_j=a_{k_s} \quad \text{and}\quad
\sum_{\substack{j\in J\\j\geq k_1}}a_j=a_J,
\]
and extracting out the same factor from the right-hand side of \eqref{e-5},
we obtain
\[
t_{k_s}-q^{a_{k_1}}t_{k_1}
=(-1)^{|J|}q^{L_{I,J}(a)}(1-q^{a_J}). \qedhere
\]
\end{proof}

By Lemma~\ref{lem-a} and Lemma~\ref{lem-b}, we can rewrite the next double sum.
\begin{prop}\label{prop-eq}
Let $I$ be a non-empty subset of $\{1,2,\dots,n\}$ of cardinality at least two. Then, for $r$ an integer
\begin{multline}\label{e-transsum}
\sum_{i\in I}\sum_{\emptyset \neq J\subseteq I\setminus \{i\}} (-1)^{|J|+1}q^{\sum_{j=i+1}^na_j+L_{I\setminus\{i\},J}(a^{(i)})}
\frac{(1-q^{a_i})(1-q^{a_{J}})}{(1-q^{|a|-a_i+r})(1-q^{|a|-a_i-a_{J}+r})}\\
=\sum_{\emptyset \neq J\subseteq I}(-1)^{|J|}q^{L_{I,J}(a)}
\frac{1-q^{a_{J}}}{1-q^{|a|-a_{J}+r}},
\end{multline}
where $a_S:=\sum_{j\in S}a_j$ for the set $S$ and $L_{I,J}(a)$ is defined as in \eqref{def-L}.
\end{prop}
\begin{proof}
By a partial fraction decomposition,
\[
\frac{1}{(1-q^{|a|-a_i+r})(1-q^{|a|-a_i-a_{J}+r})}
=\frac{1}{(1-q^{-a_{J}})(1-q^{|a|-a_i+r})}
+\frac{1}{(1-q^{a_{J}})(1-q^{|a|-a_i-a_{J}+r})}.
\]
Denote by $L$ the left-hand side of \eqref{e-transsum}. Using the above equation, we find
\begin{align*}
L&=\sum_{i\in I}\sum_{\emptyset \neq J\subseteq I\setminus \{i\}} (-1)^{|J|+1}q^{\sum_{j=i+1}^na_j+L_{I\setminus\{i\},J}(a^{(i)})}
\Big(\frac{-q^{a_{J}}(1-q^{a_i})}{1-q^{|a|-a_i+r}}
+\frac{1-q^{a_i}}{1-q^{|a|-a_i-a_{J}+r}}\Big)\\
&=\sum_{i\in I}\Big(\frac{q^{\sum_{j=i+1}^na_j}(1-q^{a_i})}{1-q^{|a|-a_i+r}}
\sum_{\emptyset \neq J\subseteq I\setminus \{i\}}
(-1)^{|J|}q^{L_{I\setminus\{i\},J}(a^{(i)})+a_{J}}\Big)\nonumber \\
&\quad
+\sum_{i\in I}\sum_{\emptyset \neq J\subseteq I\setminus \{i\}} (-1)^{|J|+1}q^{\sum_{j=i+1}^na_j+L_{I\setminus\{i\},J}(a^{(i)})}
\frac{1-q^{a_i}}{1-q^{|a|-a_i-a_{J}+r}}.
\end{align*}
We can further rewrite $L$ as
\begin{multline}\label{e-7}
L=\sum_{i\in I}\bigg(\frac{q^{\sum_{j=i+1}^na_j}(1-q^{a_i})}{1-q^{|a|-a_i+r}}
\sum_{\emptyset \neq J\subseteq I\setminus \{i\}}
(-1)^{|J|}q^{L_{I\setminus\{i\},J}(a^{(i)})+a_{J}}\bigg)\\
\quad+\sum_{\substack{J\subseteq I\\|J|>1}}
\Big(\frac{1}{1-q^{|a|-a_J+r}}
\sum_{i\in J}
(-1)^{|J\setminus \{i\}|+1}q^{\sum_{j=i+1}^na_j+L_{I\setminus\{i\},J\setminus \{i\}}(a^{(i)})}(1-q^{a_i})\Big).
\end{multline}
By Lemma~\ref{lem-a}
\begin{subequations}\label{e-main}
\begin{equation}\label{e-main1}
\sum_{\emptyset \neq J\subseteq I\setminus \{i\}}
(-1)^{|J|}q^{L_{I\setminus\{i\},J}(a^{(i)})+a_{J}}
=-q^{L_{I,\{i\}}(a)-\sum_{j=i+1}^na_j}.
\end{equation}
By Lemma~\ref{lem-b}, for $J\subseteq I$ and $|J|>1$ we have
\begin{equation}\label{e-main2}
\sum_{i\in J} (-1)^{|J\setminus \{i\}|+1}q^{\sum_{j=i+1}^na_j+L_{I\setminus\{i\},J\setminus \{i\}}(a^{(i)})}
(1-q^{a_i})
=(-1)^{|J|}q^{L_{I,J}(a)}(1-q^{a_J}).
\end{equation}
\end{subequations}
Substituting \eqref{e-main} into \eqref{e-7} yields \eqref{e-transsum}.
\end{proof}

\section{Proof of Theorem~\ref{Thm-rec}}\label{sec-proof}

In this section, we give a proof of Theorem~\ref{Thm-rec} using the inductive formula \eqref{e-ind} for $D_{v,v^+}(a)$ and Proposition~\ref{prop-eq}.

\begin{proof}[Proof of Theorem~\ref{Thm-rec}]
Let $s=|I|$ be the cardinality of $I$. We then proceed by induction on $s$.

If $s=1$, then by \eqref{e-ind} we have
\begin{equation}\label{e-s1}
D_{v,v^+}(a)= q^{\sum_{j=k+1}^na_j}\qbinom{|a|+r-1}{a_{k}-1}D_{v^{(k)},(v^{(k)})^+}(a^{(k)}),
\end{equation}
where $k$ is the index of the unique largest element in $v$.
It is easy to check that
\[
\qbinom{|a|+r-1}{a_{k}-1}=
\frac{1-q^{a_k}}{1-q^{|a|-a_k+r}}
\frac{\qbinom{|a|+r-1}{a,r-1}}
{\qbinom{|a|-a_k+r-1}{a^{(k)},r-1}}.
\]
Substituting the above equation into \eqref{e-s1} we obtain
the $s=1$ case of \eqref{formuh1}.

Suppose $2\leq s\leq n$. By \eqref{e-ind} and the induction hypothesis, we find
\begin{multline}\label{e-p1}
D_{v,v^+}(a)=
\sum_{i\in I} q^{\sum_{j=i+1}^na_j}\qbinom{|a|+r-1}{a_{i}-1}
D_{v^{(I)},(v^{(I)})^+}\big(a^{(I)}\big)\frac{\qbinom{|a|-a_i+r-1}{a^{(i)},r-1}}
{\qbinom{|a|-a_I+r-1}{a^{(I)},r-1}} \\
\times\sum_{\emptyset \neq J\subseteq I\setminus \{i\}}(-1)^{|I\setminus J|-1}q^{L_{I\setminus\{i\},J}(a^{(i)})}
\frac{1-q^{a_{J}}}{1-q^{|a|-a_i-a_{J}+r}}.
\end{multline}
Since
\[
\qbinom{|a|+r-1}{a_{i}-1}\qbinom{|a|-a_i+r-1}{a^{(i)},r-1}
=\frac{1-q^{a_i}}{1-q^{|a|-a_i+r}}\qbinom{|a|+r-1}{a,r-1},
\]
we can rewrite \eqref{e-p1} as
\begin{multline}\label{e-p2}
D_{v,v^+}(a)=(-1)^{|I|}D_{v^{(I)},(v^{(I)})^+}\big(a^{(I)}\big)\frac{\qbinom{|a|+r-1}{a,r-1}}
{\qbinom{|a|-a_I+r-1}{a^{(I)},r-1}}
\sum_{i\in I}\sum_{\emptyset \neq J\subseteq I\setminus \{i\}} (-1)^{|J|+1}q^{\sum_{j=i+1}^na_j+L_{I\setminus \{i\},J}\big(a^{(i)}\big)} \\
\times\frac{(1-q^{a_i})(1-q^{a_{J}})}{(1-q^{|a|-a_i+r})(1-q^{|a|-a_i-a_{J}+r})}.
\end{multline}
Then the theorem follows by Proposition~\ref{prop-eq}.
\end{proof}

\end{document}